\documentclass[12pt, a4paper,DIV=11]{scrartcl}

\usepackage{amsthm,amsfonts,amsmath,latexsym,mathrsfs,amssymb}
\usepackage{mathtools}

\usepackage[utf8]{inputenc}
\usepackage{xcolor}

\usepackage{libertine}
\usepackage{libertinust1math}
\usepackage[T1]{fontenc}

\usepackage{enumitem}

\usepackage{hyperref}
\hypersetup{colorlinks=true, linkcolor=blue!60!black, citecolor=blue!60!black, urlcolor=blue!60!black}

\newtheorem{theorem}{Theorem}[section]
\newtheorem{lemma}[theorem]{Lemma}
\newtheorem{result}[theorem]{Result}
\newtheorem{proposition}[theorem]{Proposition}
\newtheorem{corollary}[theorem]{Corollary}

\theoremstyle{definition}
\newtheorem{definition}[theorem]{Definition}

\theoremstyle{remark}
\newtheorem{remark}[theorem]{Remark}

\numberwithin{equation}{section}

\newcommand{\bdy}{\partial}

\newcommand{\disk}{\mathbb{D}}
\newcommand{\Cp}{\mathbb{C}_{01}}

\newcommand{\hol}{\mathcal{O}}

\newcommand{\C}{\mathbb{C}}

\newcommand{\Z}{\mathbb{Z}}
\newcommand{\D}{\mathbb{D}}

\usepackage{enumitem}

\setlist[enumerate]{
  itemsep=2pt,
  topsep=2pt,
  parsep=0pt,
  partopsep=0pt,
}

\makeatletter
\def\blfootnote{\gdef\@thefnmark{}\@footnotetext}
\makeatother

\begin{document}

\title{Revisiting Kobayashi hyperbolicity on planar domains}

\author{ Bharathi Thiruvengadam \thanks{Bharathi is supported by a fellowship
from the University Grants Commission, India.\\
  \href{mailto:212114004@smail.iitpkd.ac.in}{212114004@smail.iitpkd.ac.in}, \href{mailto:bharathit.math@gmail.com}{bharathit.math@gmail.com}\\
Department of Mathematics, Indian Institute of Technology Palakkad,
Palakkad, Kerala- 678623, India} \and Jaikrishnan Janardhanan \thanks{Jaikrishnan is supported by a grant from ANRF: \textbf{ANRF/ARGM/2025/000619/MTR} \\
\href{mailto:jaikrishnan@iitpkd.ac.in}{jaikrishnan@iitpkd.ac.in}\\
Department of Mathematics, Indian Institute of Technology Palakkad,
Palakkad, Kerala- 678623, India} }

\date{\today}

\maketitle

\begin{abstract}
We give two new elementary proofs of the complete Kobayashi hyperbolicity of the twice-punctured complex plane. We also present an extremely short proof that bounded domains are complete Kobayashi hyperbolic. Our proofs rely neither on the fact that the universal cover of the twice-punctured plane is the disk nor on the existence of negatively curved metrics. As applications, we present concise proofs of the classical theorems of Landau, Schottky, and Picard. Finally, we provide a characterization of Kobayashi hyperbolicity for planar domains inspired by a similar result of Hahn.
\end{abstract}

\blfootnote{\textup{2020} \textit{Mathematics Subject Classification}:
Primary 32Q45; Secondary 30F45, 32F45}

\blfootnote{\textit{Key words and phrases}:  Complete hyperbolicity, Normal family, Landau's theorem, Kobayashi pseudodistances}

\section{Introduction}

In 1967, Kobayashi \cite{Kobayashi1967} introduced the Kobayashi pseudodistance on complex manifolds which
has become one of the most
influential ideas in geometric function theory. The Kobayashi pseudodistance is
indispensable in the study of several complex variables, complex dynamics, complex differential geometry,
Teichmüller theory, and several other areas of mathematics. Despite its ubiquity
in the study of complex analysis in higher dimensions, the Kobayashi
pseudodistance is not often discussed in the context of planar domains. Indeed,
it rarely appears in standard textbooks on complex analysis; the classic textbook by Krantz \cite{Krantz2004book} and the expository by the same author \cite{Krantz2008kob} seem to be the only accessible sources for a beginner to learn about the Kobayashi pseudodistance on planar domains. The purpose of
this article is to revisit the Kobayashi pseudodistance on planar domains and to
highlight its utility in proving some classical theorems of complex analysis. We
require a couple of quick definitions before we can discuss our main results; see
Section~2 for more details. A complex manifold is said to be
Kobayashi hyperbolic if the Kobayashi
pseudodistance on it is a genuine distance, and the manifold is said to be complete
Kobayashi hyperbolic if the metric space thus obtained is complete.

The main objective of this article is to give two
proofs of the complete Kobayashi hyperbolicity of $\Cp$ that are 
elementary. We now answer the pertinent question of why this is a worthwhile endeavor. 
Our main inspiration comes from a paper by Hahn \cite{Hahn1980} in which he has shown that the
Kobayashi hyperbolicity of planar domains is equivalent to the classical
theorems of Schottky, Landau, and Picard. The proof is quite short and
uses only some basic properties of the Kobayashi pseudodistance. Nevertheless,
there are only a few textbooks on complex analysis that prove the classical
theorems of Picard and Schottky using the Kobayashi pseudodistance. One possible reason for this is
that Hahn's result is elementary \emph{except} for its reliance on
the complete Kobayashi hyperbolicity of the twice-punctured complex plane $\Cp:=\C \setminus \{0,1\}$.
The standard way to see this is to observe that
Kobayashi hyperbolicity and completeness are preserved under covering maps and
that the universal covering space of $\Cp$ is the unit disk $\D$. Although this is
a completely standard fact by the time one is typically introduced to the Kobayashi
distance, it is definitely \textbf{not} a standard fact in a first (or maybe even second?)
course on complex analysis. 

Another way to see the hyperbolicity of $\Cp$ is to
use the fact that we can construct a negatively curved conformal metric on
$\Cp$. From the classical Ahlfors--Schwarz lemma (\cite{Ahlfors1938}), it is easy to prove
that any Riemann surface that admits a Hermitian
metric whose curvature is bounded above by a negative constant is Kobayashi hyperbolic.
It is certainly possible to present the construction of such a metric on
$\Cp$ in an elementary
manner. Nevertheless, there are a few reasons why one might want a direct proof
of the complete hyperbolicity of $\Cp$. First, the construction of such a metric, though elementary, is not
entirely transparent. The famous book by Krantz \cite{Krantz2004book} gives the
simplest such metric:
\[
\mu(z)=\left[\frac{\left(1+|z|^{1 / 3}\right)^{1 / 2}}{|z|^{5 / 6}}\right] \cdot\left[\frac{\left(1+|z-1|^{1 / 3}\right)^{1 / 2}}{|z-1|^{5 / 6}}\right],
\]
and contains a brief discussion on \textit{``where this non-intuitive definition came from''}; see \cite[Theorem~3, Chapter~2]{Krantz2004book} and the discussion after the proof. 
Second, the completeness of the Kobayashi distance on
$\Cp$ is not immediate from this approach, and another argument is still needed to show completeness. In fact, the metric $\mu$ above is \textbf{not} complete. The existence of a complete conformal metric on $\Cp$ was given by Grauert and Reckziegel in \cite{Grauert1965hermitian}. A modern construction, along with a detailed proof, of a complete conformal metric can be found in Chapter~4 of the book \cite{Narasimhan2001} and is very tedious. Our proof 
that $\Cp$ is complete Kobayashi hyperbolic (see Theorem~\ref{T:compzalc} below) is much shorter and more straightforward.  Finally, we believe that the complete Kobayashi hyperbolicity of $\Cp$ is of sufficient
importance that it is worth recording one or two new proofs of this fact. Indeed, it is our
hope that the grand equivalence we establish in the final section is evidence of this claim. 

The main objective of this article is to give two
proofs of the complete Kobayashi hyperbolicity of $\Cp$ that are truly
elementary. This is the content of Section~3 and Section~5; see the
comments there for more about the nature of these proofs. Using complete hyperbolicity, in
Section~4, we give what are perhaps the shortest proofs
of the classical theorems of Landau, Montel, Picard, and Schottky. Finally, in Section~6, we
prove a more general version of the main result in Hahn's paper, proving that the
classical theorems of Landau, Montel, Picard, and Schottky are all equivalent to
Kobayashi hyperbolicity. We emphasize again that, unlike Hahn's proof, our proof does \textbf{not} use the fact that the universal cover of $\Cp$ is the disk.

\subsection*{Notations}
We fix some notation that will be used throughout the article. 
\begin{enumerate}
\item $\D = \{ z \in \C : |z| < 1 \}$ denotes the unit disk in $\C$, and
  $\D^* = \D \setminus \{0\}$ denotes the punctured unit disk in $\C$. 
  
  \item $\D(a,r)$ denotes the Euclidean disk in $\C$ with center $a$ and radius $r>0$.
  \item $\mathbb{H}$ denotes the upper half-plane, defined as $\{ z \in \C : \Im(z) > 0 \}$. 
  
  \item $\Cp$ denotes the twice punctured complex plane $\C \setminus \{0,1\}$.
  \item $\mathcal{C}(M,N)$ denotes the set of all continuous maps between complex manifolds $M$ and $N$, and $\hol(M,N)$ denotes the set of all holomorphic maps from $M$ to $N$.
  
\end{enumerate}

\section{Kobayashi hyperbolicity}
In this section, we briefly present the definitions of the Poincaré distance and the Kobayashi
pseudodistance, along with some important properties
of these distances that are used in this article. As the material in this
section is standard, we will not give precise references for each result. Please refer to
\cite{Abate2023book,SKobayashi1998,jp2013} for more details on the Poincaré
distance and the Kobayashi pseudodistance.

The Poincaré distance is a fundamental concept in complex analysis. It is defined on the unit disk $\D$ and is given by the formula
\[
\omega(z, w) = \tanh^{-1} \left| \frac{z - w}{1 - \bar{w} z} \right| =\frac{1}{2}\log \left( \frac{1+ \left| \dfrac{z - w}{1 - \bar{w} z} \right|}{1-\left| \dfrac{z - w}{1 - \bar{w} z} \right|} \right) \quad \text{for}\ z,w \in \D.
\] 

The \textit{Schwarz-Pick Lemma} can be interpreted as saying that the
holomorphic self maps of $\D$ are distance-decreasing with respect to the
Poincaré distance $\omega$. In fact, automorphisms are isometries with respect
to $\omega$.  The Poincaré distance $\omega$ is a complete distance on $\D$ that induces the Euclidean topology on $\D$. 

We define the Kobayashi pseudodistance on a complex manifold $M$ using the Poincaré distance $\omega$. 

\begin{definition}[Kobayashi pseudodistance] Let $M$ be a complex manifold. Let $p, q \in M$. We choose points $p = p_0,\ldots,p_k = q$ of $M$, points 
  $a_1,\ldots,a_k,b_1,\ldots,b_k$ of the unit disk, and holomorphic mappings
  $f_1,\ldots,f_k$ from $\disk$ into $M$, such that $f_i(a_i) = p_{i-1}$ and
  $f_i(b_i) = p_i$ for $i = 1,\ldots,k$. For each such choice of points and
  mappings, we consider the number
  \[
	\sum_{i=1}^k \omega(a_i,b_i).
  \]
  The \textit{Kobayashi pseudodistance} $d_M(p,q)$ is the infimum of the set of all such
  numbers obtained by varying the choices of points and mappings.
\end{definition}

It is not hard to show that the Kobayashi pseudodistance on a complex manifold
is, indeed, a pseudodistance. The most important property of the Kobayashi
pseudodistance is its contractivity under holomorphic maps.

\begin{result} \label{thm: distance decreasing}
  If $f: M \to N$ is a holomorphic map between complex manifolds $M$ and $N$, then 
    \[
    d_N(f(p), f(q)) \leq d_M(p, q)\quad \text{for all } p, q \in M.
    \] 
    In particular, if $f$ is a biholomorphic map, then 
    \[
    d_N(f(p), f(q)) = d_M(p, q)\quad \text{for all } p, q \in M.
    \]
\end{result} 

The next property is that the Kobayashi pseudodistance is the largest pseudodistance on $M$ that is distance-decreasing under holomorphic maps. Consequently, the Kobayashi pseudodistance coincides with the Poincaré distance on the unit disk $\D$.

\begin{result} \label{pro: largest}
  Let $M$ be a complex manifold and $d$ be a pseudodistance on $M$. Suppose that for any $f \in \hol(\D,M)$, we have 
  \[
  d(f(z),f(w))\leq \omega(z,w)\quad \text{for}\ z,w\in \D.
  \] 
  Then $d\leq d_M$.
\end{result} 

Now, we define the notion of Kobayashi hyperbolicity and completeness. 

\begin{definition}
  A complex manifold $M$ is said to be \textbf{Kobayashi hyperbolic} if the Kobayashi
  pseudodistance $d_M$ is a genuine distance on $M$. Furthermore, $M$ is said to
  be \textbf{complete Kobayashi hyperbolic} if the metric space $(M,d_M)$ is complete.
\end{definition}

\begin{result}
  The topology induced by the Kobayashi pseudodistance on a complex manifold $M$
  coincides with the manifold topology on $M$ whenever $M$ is Kobayashi
  hyperbolic. Furthermore, $M$ is Kobayashi complete hyperbolic if and only if
  every closed ball in $(M,d_M)$ is compact. 
\end{result}

Finally, there is a nice relationship between the Kobayashi pseudodistances whenever we 
have a holomorphic covering map between two complex manifolds. 

\begin{proposition} \label{pro: under covering}
  Let $M$ be a complex manifold, and $(\widetilde{M},\pi)$ be a covering manifold.
  Let $p,q \in M$ and $\widetilde{p}\in \widetilde{M}$ be such that $\pi(\widetilde{p}) = p$. Then
  \[
	d_M(p,q) = \inf_{\widetilde{q} \in
	\pi^{-1}(q)}d_{\widetilde{M}}(\widetilde{p},\widetilde{q}).
  \]
  Furthermore, $M$ is hyperbolic if and only if $\widetilde{M}$ is hyperbolic, and $M$ is complete if and only if $\widetilde{M}$ is complete.
\end{proposition}

\begin{corollary}
  The punctured disk $\D^*$ and annuli $\{ z \in \C : r < |z| < 1 \}$ for any $0 < r < 1$ are complete hyperbolic.
\end{corollary}

One of the most profound consequences of complete hyperbolicity is the following
general version of Montel's theorem. We need the definition of a normal family.

\begin{definition}[Normal family]
  Let $M$ and $N$ be complex manifolds. A family $\mathcal{F} \subset
  \hol(M,N)$ is said to be \textbf{normal} if every sequence in $\mathcal{F}$
  either has a subsequence that converges in the compact-open topology to a
  holomorphic map from $M$ to $N$ or the entire sequence diverges compactly; that is, for every
  compact subset $K$ of $M$ and every compact subset $L$ of $N$, there exists an
  integer $N_{K,L}$ such that for all $n > N_{K,L}$, $f_n(K) \cap L =
  \emptyset$. Here, $\{f_n\}$ is the sequence under consideration.
\end{definition}

\begin{result}[Montel's theorem]
  Let $N$ be a complete hyperbolic complex manifold and $M$ be a complex
  manifold. Then the family $\hol(M,N)$ is normal.
\end{result}

We now define the Kobayashi--Royden infinitesimal pseudometric. The proofs of results in the rest of this section can be found in \cite{Royden1971}.

\begin{definition}\label{def: Kobayashi infinitesimal pseudometric}
    Let $M$ be a complex manifold, $p \in M$, and $v \in T_pM$. The \textbf{Kobayashi infinitesimal pseudometric at $p$ and tangent to the direction $v$}, denoted by $F_M(p;v)$, is defined as
    \[
    F_M(p;v) = \inf \left\{ r >0 \mid \exists\, f \in \hol(\D,M)\ \text{such that}\ f(0) = p,\ r f'(0)(1) = v \right\}.
    \] 
\end{definition} 

The following properties of the Kobayashi--Royden infinitesimal pseudometric are
relevant to this article.

\begin{result} \label{pro: properties of infinitesimal}
  Let $M$ be a complex manifold. Then the following holds:
  \begin{enumerate}
    \item $F_M(p;cv) = |c| F_M(p;v)$ for all $c \in \C$ and $v \in T_pM$.
    \item $F_M$ is upper semicontinuous on $TM$. 
    \item If $g: M \to N$ is a holomorphic map between complex manifolds $M$ and $N$, then 
    \[
    F_N(g(p); g'(p)(v)) \leq F_M(p; v)\quad \text{for all } p \in M, v \in T_pM.
    \] 
    In particular, if $g$ is a biholomorphic map, then 
    \[
    F_N(g(p); g'(p)(v)) = F_M(p; v)\quad \text{for all } p \in M, v \in T_pM.
    \]
  \end{enumerate}
\end{result}

\begin{remark}
  The Kobayashi--Royden pseudometric for the disk is nothing but the Poincaré
  metric:
  \[
    F_\D(z;v) = \frac{|v|}{1 - |z|^2},\quad \text{for all } z \in \D, v \in T_z\D \cong \C.
  \]
    An easy computation shows that the Kobayashi--Royden pseudometric for the
  punctured disk $\D^*$ is given by
  \[
    F_{\D^*}(z;v) = \frac{|v|}{-2|z|\ln|z|},\quad \text{for all } z \in \D^*, v \in T_z\D^* \cong \C.
  \]
\end{remark}

\begin{definition}
  Let $M$ be a complex manifold. We define the \textit{length} of a $C^1$ smooth curve $\sigma: [0,1] \to M$ as 
\[
\ell_k(\sigma) := \int_0^1 F_M(\sigma(t); \sigma '(t)) dt.
\] 
And define a new pseudodistance on $M$ as \[K_M(z,w)=\inf \ell_k(\sigma)\] where the infimum is taken over all $C^1$ smooth curves $\sigma: [0,1] \to M$ with $\sigma(0) = z$ and $\sigma(1) = w$. 
\end{definition}

It is not difficult to show that $K_M$ is a pseudodistance on $M$, that holomorphic maps
 decrease in distance with respect to $K_M$, and that $K_\D = \omega$. Then we have the following result.

\begin{result}
  $K_M = d_M$
\end{result}

The next result gives us a way to determine the hyperbolicity of a complex
manifold using its infinitesimal pseudometric.

\begin{result} \label{L:hypinfinitesimal}
Let $D$ be a domain in $\C^n$. The domain $D$ is Kobayashi hyperbolic iff for
each $ p \in D $, there exists a neighborhood $ U $ of $ p $ and a positive
constant $ c $ such that $F_D(z; v) > c |v|$ for all $ z \in U $ and $ v \in
T_zD=\C^n $. 
\end{result}

\section{Complete hyperbolicity of $\Cp$ using rescaling}

We now give a direct proof of the complete Kobayashi hyperbolicity of $\Cp$.
The most sophisticated result used in our proof is a version of the Arzela--Ascoli
theorem. We are, therefore, justified in calling this proof truly elementary! Our proof of Kobayashi
hyperbolicity uses a rescaling argument. Similar arguments are used in the proof
of the famous Brody's lemma \cite{Brody1978} as well as in the proof of
Zalcman's formulation of Bloch's principle \cite{Zalcman1975}. Zalcman credits the idea to the
paper of Lohwater and Pommerenke \cite{LohwaterPommerenke1973}. Zalcman's ideas have been vastly
extended in the recent work of Berteloot \cite{Berteloot2025Zalc}.

Our proof is modeled on the crisp 
proof of Brody's lemma in \cite{Duval2021Around} combined with the idea of
using $n$-th roots to contradict Liouville's theorem due to Ros \cite[Theorem~2]{Ros2002}. Ros's $n$-th root trick has been used by Zalcman to give a proof of Montel's theorem \cite{Zalcman1998}. We note that
Montel's theorem immediately gives hyperbolicity by a basic fact of Kiernan \cite{Kiernan1970}. However, Zalcman's proof of Montel's theorem uses both Marty's theorem on normality
as well as Zalcman's own famous rescaling normality lemma. Our direct proof below does \textbf{not} use Marty's theorem and relies only on the basic idea in 
Zalcman's lemma and Ros's $n$-th root trick. Once hyperbolicity
is established, completeness follows from a very simple argument that we found in Kobayashi's
book \cite[Lemma~3.3.5]{SKobayashi1998}. We need the following well known result about normal
families.

\begin{result}[\cite{Wu1967}, Lemma~1.1] \label{R:normalfamily}
Let $\mathcal{F} \subseteq \mathcal{C}(M, N)$, where $M, N$ are connected,
locally compact metric spaces. Then
\begin{enumerate}
  \item If $\mathcal{F}$ is compact, then $\mathcal{F}$ is normal.
   \item If $\mathcal{F}$ is normal, then its closure is locally compact.
   \item If $\mathcal{F}$ is equicontinuous and if each bounded subset of $N$ is relatively compact, then $\mathcal{F}$ is normal.
\end{enumerate}

\end{result}

\begin{theorem}\label{T:compzalc}
  The domain $\Cp$ is complete Kobayashi hyperbolic. 
\end{theorem} 

\begin{proof} 
    We will show that we find a constant $c > 0$ such that for each $z \in \D^*$, we have 
    \[
      F_{\Cp}(z;1) > c.
    \]
    If the above is not true, then it is clear that we can find a sequence $z_n
    \in \D^*$ and holomorphic maps $f_n:\D \to \Cp$ such that
    $f_n(0) = z_n$ and $|f_n'(0)| > n^2$. We replace each $f_n$ by some branch of
    $f_n^{1/n}$ so that we have $|f_n'(0)| > n$ and each $f_n$
    misses the $n$-th roots of unity.     
    
    We may assume that $f_n$ is smooth up to the boundary by slightly shrinking
    the disk and scaling. Consider the function
    \[
      g_n(z) = \delta(z)|f_n'(z)|,
    \]
    where $\delta(z) := \textsf{dist}(z,\partial\D)$ is the distance from $z$ to
    $\bdy \D$. This is a continuous function on $\overline{\D}$ that vanishes on
    the boundary. Let $a_n \in \disk$ be a point in $\D$ such that the maximum of
    $g_n$ is attained at $a_n$. Let $D_n := \D(a_n, \delta(a_n)/2)$. For $z \in
    D_n$, we have $\delta(z) \geq \delta(a_n)/2$ and hence
    \[
      |f_n'(z)| \leq \frac{g_n(z)}{\delta(z)} \leq \frac{g_n(a_n)}{\delta(a_n)/2} = 2|f_n'(a_n)|.
    \]

    Now, define the entire maps 
    \[
      r_n(z) = a_n + \frac{z}{f_n'(a_n)},
    \]
    and let $D_n'$ be the preimage of $D_n$ under $r_n$. As
    $\delta(a_n)|f_n'(a_n|) \geq g_n(0) = |f_n'(0)| > n$, it follows that $D_n'$
    become larger and larger disks as $n \to \infty$. Furthermore, we have 
    $|(f_n \circ r_n)'(z)| \leq 2, z \in D_n'$ and $(f_n \circ r_n)'(0) = 1$. It
    follows from a straightforward argument using Result~\ref{R:normalfamily} that we can
    find a subsequence of $f_n \circ r_n$ that
    converges uniformly on compact subsets of $\C$ to a meromorphic map
    $F:\C \to \hat{\C}$. Notice that by our construction, each $f_n
    \circ r_n$ misses the $n$-th roots of unity. Therefore, $F$ misses the entire
    unit circle which is impossible. Thus, we conclude that
    there is a positive constant $c$ such that
    \[
      F_{\Cp}(z;1) > c, \quad \text{for all } z \in \D^*. 
    \]
    It is now clear that the hypothesis of Result~\ref{L:hypinfinitesimal} is
    satisfied whenever $z \in \D^*$. Result~\ref{pro: properties of infinitesimal} says $F_{\Cp}$ is invariant under the automorphism 
    $1/z$ of $\Cp$ and that $F_{\Cp}$ is upper semi-continuous. 
    We see that the hypothesis of Result~\ref{L:hypinfinitesimal} is satisfied
    at each point of $\Cp$. Therefore, $\Cp$ is hyperbolic.

    To see that the Kobayashi distance on $\Cp$ is complete, let $z_n$ be a
    Cauchy sequence in $\Cp$ with respect to the Kobayashi distance. It suffices to show that a
    subsequence of $z_n$ converges to a point in $\Cp$. By compactness of the Riemann sphere and by
    passing to to a subsequence, if necessary, we may assume that $z_n$ converges to a point in the Riemann
    sphere in the Euclidean topology.
    
    If the sequence $z_n$ were to converge in the Euclidean metric to a point $z_0 \in \Cp$, then
    we are done as the topology induced by the Kobayashi distance and the
    Euclidean distance coincide. As the automorphism group of $\Cp$ are
    Möbius transformations that permute $0,1$ and $\infty$, we may assume, by
    passing to a subsequence, that $z_n \in \D^*/2 := \D(0,1/2) \setminus \{0\}$. 

     Let $\delta > 0$ be chosen such that
    \[
      d_{\Cp}(z,w) > \delta \quad \text{for all } z \in \Cp \setminus \D^*, w \in \D^*/2.
    \]
    To see that such a $\delta$ exists, we first note that hyperbolicity guarantees that the 
    Kobayashi--Royden metric is continuous.  Consider the closed annulus $A$ of 
    inner-radius $2/3$ and outer-radius $3/4$. The Kobayashi--Royden metric must attain a 
    positive minimum on $A$. It follows that any curve joining the boundary circles of $A$ must
    have length bounded below by some positive constant. 
    Since any curve joining $z \in \Cp \setminus \D^*$ to $w \in \D^*/2$ must pass through this compact annulus, 
    the existence of $\delta$ is clear. 
    We may, of course, further assume that 
    $d_{\Cp}(z_n,z_m) < \delta' := \delta/2$ for all $n,m$. Let $p,q \in \D^*/2$ with
    $d_{\Cp}(p,q) < \delta'$ and let $f_i:\disk \to \Cp, i=1,\ldots,k$ be a chain
    of holomorphic maps that connect $p$ to $q$ as in the definition
    of the Kobayashi pseudodistance. More precisely, we have points $p =
    p_0,\ldots,p_k = q$ of $\Cp$, points $a_1,\ldots,a_k \in \D$ such that $f_i(0)
    = p_{i-1}$ and $f_i(a_i) = p_i$ for $i=1,\ldots,k$, and
    \[
      \sum_{i=1}^k \omega(0,a_i) < \delta'.
    \]
    It is now clear that the intermediate points $p_i$ are
    at most $\delta'$ distance away from $p$ and from each other. Therefore, each $a_i
    \in \D(0, \tanh(\delta'))$ and the image of this disk under each $f_i$ lies in $\D^*$. Define $D
    := \D(0,\tanh(\delta))$ and $D':= \D(0, \tanh(\delta'))$. Then $D' \Subset D$
    and we can find $c > 0$ such that
    \[
      \omega(0,z) \geq cd_D(0,z) \quad \text{for all } z \in D'.  
    \]
    We now have
    \[
      \sum_{i=1}^k \omega(0,a_i) \geq c \sum_{i=1}^k d_D(0,a_i) \geq c \sum_{i=1}^k d_{\D^*}(f_i(0), f_i(a_i)) \geq c d_{\D^*}(p,q).
    \]
    This is valid for any chain of holomorphic maps connecting $p$ and $q$
    whose net length is less than $\delta'$. Therefore, we have 
    \[
      d_{\Cp}(p,q) \geq c d_{\D^*}(p,q).
    \]
    It follows that $z_n$ is a Cauchy sequence in $\D^*$ with respect to the
    Kobayashi distance. But the Kobayashi distance on $\D^*$ is complete, so $z_n$
    must converge to a point in $\D^*$. Thus, we conclude that the Kobayashi distance on $\Cp$ is complete.

\end{proof}
\begin{corollary}
  Any hyperbolic planar domain $D$ is complete hyperbolic.
\end{corollary}

\begin{proof}
  Without loss of generality, we may assume that $D \subset \Cp$ as $\C$ and
  $\C\setminus \{z_0\}$ for any $z_0 \in \C$ are not hyperbolic. By the
  distance-decreasing property of the Kobayashi distance, we have
  \begin{itemize}
    \item The domain $D$ is Kobayashi hyperbolic as $\Cp$ is Kobayashi hyperbolic.
    \item Any Cauchy sequence $z_n$ in $D$ with respect to the Kobayashi distance is also a Cauchy sequence in $\Cp$ with respect to the Kobayashi distance. Since the Kobayashi distance on $\Cp$ is complete, it follows that any Cauchy sequence in $D$ converges to a point in $\Cp$.
  \end{itemize}
  Let $z_n$ be a Cauchy sequence in $D$ with respect to the Kobayashi distance. It suffices to show that some subsequence of $z_n$ converges to a point in $D$. By compactness, there is a subsequence $z_{n_k}$ that converges to a point $z\in \overline{D}$ in the Euclidean topology. Here the closure is taken in the Riemann sphere. 
   If $z\in D$, we are done. 
  Let us assume that $z \in \partial D$. Now consider the new domain $G=\C \setminus \{z,1\}$. 
  It is obvious that $D$ is an open subset of $G$ and that the sequence $(z_n)$ is also a Cauchy sequence in $(G, d_G)$. 
  Since the Kobayashi distance on $G$ is complete, it follows that $(z_n)$ converges to a point in $(G, d_G)$. But this contradicts the uniqueness of the limit points of the convergent subsequence $(z_{n_k})$. 
  Therefore, the Kobayashi distance on $D$ is complete.
\end{proof}

The same idea gives a very simple proof of the completeness of the Kobayashi
distance on bounded planar domains that does not seem to have appeared in the literature. 
\begin{theorem}
\label{thm: completeness of bounded domains}
Any bounded planar domain is complete hyperbolic.
\end{theorem}
\begin{proof}
  Let $D$ be a bounded domain in $\C$. It is clear that $D$ is a Kobayashi hyperbolic domain.
  Let $z_n$ be a Cauchy sequence in $D$ with respect to the Kobayashi distance. It suffices to show that some subsequence of $z_n$ converges to a point in $D$. By compactness, there is a subsequence $z_{n_k}$ that converges to a point $z\in \overline{D}$ in the Euclidean topology. 
  If $z \in D$, then we are done. Let us assume that $z \in \partial D$. Let
  $D'$ be large disk that contains the closure of $D$. Consider the new domain
  $G = D' \setminus \{z\}$. Since $D$ is an open subset of $G$, it is clear that
  $z_{n_k}$ is also a Cauchy sequence in $(G, d_G)$. However, since the Kobayashi distance
  on $G$ is complete, it follows that $z_{n_k}$ converges to a point in $(G,
  d_G)$, which is not possible. Therefore, the Kobayashi distance on $D$ is complete.
\end{proof}

\section{The classical theorems via complete hyperbolicity}

We now present what are perhaps the shortest proofs of Landau's theorem, Schottky's theorem, and
Picard's great theorem using the completeness of the Kobayashi distance on
hyperbolic planar domains. These proofs are certainly not new or original, but
they are quite elegant and extremely short, so we could not resist including them here as we need the statements in any case to formulate our main result in Section~6.  We
begin with Landau's theorem and note that the proof does not use the completeness of the Kobayashi distance.

\begin{theorem}[Landau's theorem]\label{T:landau}
  Let $a \in \Cp$ and let $0 \neq b \in \C$. There exists a universal constant
  $R(a,b) > 0$ such that if $f:\disk(0,R(a,b)) \to \C$ is holomorphic, with
  $f(0) = a$ and $f'(0) = b$, then $f(\disk(0,R(a,b)))$ contains either $0$ or $1$. 
\end{theorem}

\begin{proof}
  Set $R(a,b) = \dfrac{2}{F_{\Cp}(a;b)}$ and Let $f:\D(0,R(a,b)) \to \C$ be holomorphic with $f(0) =
  a$ and $f'(0) = b$. Suppose the image of $f$ does not contain either $0$ or
  $1$. Then $f$ is a holomorphic map from $\D(0,R(a,b))$ to $\Cp$. 
  Then, by the contractivity of the Kobayashi--Royden metric, we have
  \[
    F_{\Cp}(a;b) \leq \frac{1}{R(a,b)} = \frac{F_{\Cp}(a;b)}{2},
  \]
  which is a contradiction. Thus, the image of $f$ must contain either $0$ or
  $1$. 
\end{proof}

Now we come to Montel's theorem. If $D \subset \C$ is complete Kobayashi hyperbolic,
then the domain $D$, equipped with the Kobayashi distance, satisfies all the hypotheses
of Result~\ref{R:normalfamily}. Therefore, the family $\hol(\D,D)$ is normal.

Let $D$ be a hyperbolic planar domain and $a\in D$. For $R>0$, define a family $\mathcal{F}(a,R)$ by 
\[
  \mathcal{F}(a,R)=\left\{f \in \hol(\D, D) : d_D(a, f(0))<R \right\}.
\] 
Equivalently, $\mathcal{F}(a,R)$ is the family of holomorphic functions $f:\D
\to D$ such that $f(0) \in B_D(a,R)$, where $B_D(a,R)$ is the Kobayashi ball
centered at $a$ with radius $R$ in the domain $D$. We will now prove a version
of Schottky's theorem for the family $\mathcal{F}(a,R)$. Our proof is modeled on
the one in \cite{Hahn1980}

\begin{theorem}[Schottky's theorem]
   For any $0<r<1$, there is a constant $M(r,R)>0$ such that if $f\in \mathcal{F}(a,R)$, then 
  \[
  |f(z)|\leq M(r,R) \quad \text{for all}\  \ |z| \leq r.
  \] 
\end{theorem}

\begin{proof} Set $\delta=\tanh^{-1}(r)$.
 Consider the Kobayashi ball $B_D(a,R+\delta)$ centered at $a$ with radius $R+\delta$ in the domain $D$. Note that for each $f\in \mathcal{F}(a,R)$ and for all $z$ with $|z| \leq r$, we have
 \[ 
 d_D(a,f(z)) \leq d_D(a,f(0)) + d_D(f(0), f(z)) < R + \omega(0,z) < R + \delta,
 \]
  this shows that $f(z) \in B_D(a,R+\delta)$ for all $|z| \leq r$. Since the Kobayashi distance on $D$ is complete, the ball $B_D(a,R+\delta)$ is relatively compact in $D$. Therefore, there is a constant $M(r,R)>0$ such that $|f(z)| \leq M(r,R)$ for all $|z| \leq r$. This completes the proof.
\end{proof}

The proof of Picard's great theorem below recasts the
one given in \cite{Lvovski2020} in terms of the completeness of the Kobayashi
distance. We believe that Lvovski's proof is inspired by
the ideas from Kwack's famous generalization of Picard's theorem
\cite{Kwack1969} and the elementary proof in \cite{Bridges1981}. 

\begin{theorem}[Great Picard Theorem]
  \label{thm: Big Picard}
  Let $f: \D^* \to \Cp$ be a holomorphic function. Then $0$ is either a removable singularity or a pole of $f$. Consequently, suppose  that $0$ is an essential singularity of $f:\D^* \to \C$, then $f$ omits at most one value in $\C$.
\end{theorem} 
\begin{proof}
  Suppose $z_0$ is an essential singularity. By the Casorati--Weierstrass theorem, we can find a sequence $z_n \to 0$, with $|z_n|$ strictly decreasing, in $\D^*$ such that $f(z_n) \to a \in \Cp$.
  We may assume that $|z_n| < 1/2$ for all $n$ and that $f(z_n)
  \in B_{\Cp}(a,R)$ (ball in the Kobayashi distance) for some $R > 0$. If $\sigma$ is a circle centered at the origin with radius  $r < 1/2$, then we have
  \[
    \ell_k(\sigma) = \int_0^1 F_{\D^*}(\sigma(t); \sigma'(t)) dt = \int_0^1 \frac{2\pi r}{-2r\ln r} dt = \frac{\pi}{-\ln r}<\frac{\pi}{\ln 2}. 
  \]
  This shows that $f(z)$ lies in a Kobayashi ball centered at $a$ with radius
  $\delta = R +\pi/\ln 2$ whenever $|z| = |z_n|$ for some $n$. Note that, by
  completeness, this ball is precompact in $\Cp$. By the Casorati--Weierstrass theorem again, we can find points in between these circles that are mapped
  outside any large disk centered at the origin. This contradicts the maximum principle, and hence $0$ cannot  be an essential singularity.
\end{proof}

\section{Complete hyperbolicity of $\Cp$ using Landau's theorem}

The goal of this section is to show that $\Cp$ is complete Kobayashi hyperbolic
using Landau's theorem. We begin by stating the Bloch--Landau theorem which was
used by Landau to give the first ``elementary'' proof of Picard's great theorem. As 
Kobayashi hyperbolicity implies Landau's theorem, it follows that the Bloch--Landau theorem 
implies Landau's theorem.

\begin{result}[Bloch--Landau theorem]
  Let $f$ be a holomorphic function on $\D(z_0,r)$ and $f'(z_0) \neq 0$. Then
  the image $f(\D(z_0,r))$ contains the disk $\D(w,c|f'(z_0)|r)$, for some
  absolute universal constant $c > 0$ and some point $w \in \C$.
\end{result}

\begin{remark}
  We again emphasize that this $c$ is independent of the holomorphic function
  $f$. Landau's constant is defined as the supremum of all such constants $c$
  for the case when $z_0 = 0$ and $r = 1$. The exact value of Landau's constant
  is not known. 
\end{remark}

\begin{remark}
  Bloch originally proved this theorem in a stronger form, where he showed that
  the image $f(\D(z_0,r))$ contains a schlicht disk of radius $c|f'(z_0)|r$; that is,
  there exists a subdomain of $\D(z_0,r)$ on which $f$ is injective and whose
  image is a disk of radius $c|f'(z_0)|r$. Here $c$ is a universal constant, and
  the supremum of all such constants is known as Bloch's constants when $z_0 =
  0$ and $r = 1$. The exact value of Bloch's constant is also not known.
\end{remark}

Our strategy for proving the hyperbolicity of $\Cp$ is inspired
by Landau's elementary approach to proving Picard's theorem using the above theorem. Towards the
completion of this manuscript, we found that Lvovski, in his recent book
\cite{Lvovski2020}, has used the exact same strategy to prove the hyperbolicity of
$\Cp$. However, Lvovski deals only with the infinitesimal metric
and does not explicate his results for the Kobayashi distance, so we have elected
to give a detailed proof here as we need these results for our main theorem in
Section~6. The outline of our proof is as follows.

\begin{enumerate}
  \item Given a holomorphic map $f:\D \to \Cp$, we construct a
  holomorphic map $k:\D \to \C$ that omits a lattice in $\C$. There are several
  constructions of such a map available in the literature. 
  See \cite[Theorem~2.2]{Segal2007nine} for a very detailed and well-motivated
  construction. The proofs of Picard's theorems in the textbooks by Remmert
  \cite{Remmert1998ClassicalTopics} and Conway \cite{Conway1978FOOCV1} also
  contain constructions of such maps.
  We have borrowed, and slightly modified, the construction from \cite[Exercise~10.14]{Marshall2019}.
  The one used by Lvovski in \cite{Lvovski2020} is slightly different. 
  \item Using the structure of this lattice, we show that there exists a
  constant $R > 0$ such that for any $z \in \C$, $k(\D)$ does not contain the
  disk $\D(z,R)$.
  \item By the Bloch--Landau theorem and chain rule, we obtain an upper bound on
  $|f'(0)|$ in terms of an absolute constant and $|f(0)|$.
  \item Finally, we conclude that there exists a neighborhood $U$ of any point
  $z_0$ in $\Cp$, and a positive constant $c$ such that 
  \[
    F_{\Cp}(z;1) > c, \quad \text{for all } z \in U.
  \]
\end{enumerate}

The following Lemma constructs the required holomorphic map $k$ that omits a
lattice in $\C$. 

\begin{lemma}\label{lem: branch of logarithm}
  Let $D$ be a simply connected domain in $\C$, and $f$ be a holomorphic map on
  $D$ that omits the values $0$ and $1$. We can construct a holomorphic map $k:D
  \to \C$ such that the image of $k$ does not contain any disk of radius greater than $2$.  
\end{lemma}

\begin{proof}
  Since $D$ is simply connected and $f$ omits the value $0$, we can find a
  branch of the logarithm of $f$ on $D$, say $g$. Since $f$ omits $1$, the map
  $g$ omits the values of the form \[2n\pi i,\ n \in \Z.\]
  Therefore, the map $h := \frac{i}{2\pi}g$ omits the integers.
  In the same way, we can define a holomorphic map $k:D \to \C$ such that
  \[
    k(z) = \frac{-i}{\pi} \log \left( h(z) + \sqrt{{h
    (z)}^2 - 1} \right).
  \]
  Elementary algebra shows that the map $h + \sqrt{h^2 - 1}$ must omit values of
  the form $n + \sqrt{n^2 - 1},n \in \Z$. This means the map $k$ omits the
  values of the form 
\[
  m-i\log\left(\Big|n+\sqrt{n^2-1}\Big|\right),\  m ,n \in \Z.
\]  
  Note that the number
  \[\log\left(\Big|n+1+\sqrt{(n+1)^2-1}\Big|\right)-\log\left(\Big|n+\sqrt{n^2-1}\Big|\right)\]
  is less than or equal to
  \[\log\left(2n+2\right)-\log\left(2n-1\right)=\log\left(\frac{2n+2}{2n-1}\right)\leq
  \log(4).\] 
  This means that the image of $k$ does not contain any disk of radius greater than $2$.  
\end{proof}

\begin{theorem}
  The domain $\Cp$ is complete hyperbolic. 
\end{theorem} 

\begin{proof} 
    Let $z_0 \in \Cp$. By Result~\ref{L:hypinfinitesimal}, it suffices to show that there exists a neighborhood $U$ of $z_0$ and a positive constant $c$ such that
    \[
      F_{\Cp}(z;1) > c, \quad \text{for all } z \in U.
    \]
    Consider a holomorphic map $f:\D \to \Cp$ with $f(0)=z_0$. 
    By the previous lemma, we can construct the associated holomorphic map $k:\D
    \to \C$ such that the image of $k$ does not contain any disk of radius
    greater than $2$. The Bloch--Landau theorem now shows that there is an
    absolute bound on $|k'(0)|$ independent of the choice of $f$. In fact this
    bound is even independent of the choice of $z_0$. Notice that the derivative
    of $h$ at $0$ is just some constant multiple of the derivative of $f$ at $0$
    with constant depending on $z_0$. More precisely, we have
    \[
      h'(0) = \frac{i}{2\pi} \cdot \frac{f'(0)}{f(0)} = \frac{i}{2\pi} \cdot
      \frac{f'(0)}{z_0},
    \]
    which means $f'(0) = C_1 z_0 h'(0)$ for some constant $C_1$. Similarly, we have
    \[
        k'(0) = \frac{-i}{\pi} \, \frac{h'(0)}{\sqrt{\,h(0)^2 - 1\,}}.
    \]
    And the derivative of $k$ at $0$ is again
    some constant multiple of the derivative of $h$ at $0$ with constant
    depending on $z_0$. This shows that there exists a positive constant $C$
    such that
    \[
    |f'(0)| \leq C|z_0| \left|\sqrt{(\log(z_0))^2 - 1}\right|.
    \]   
    Note that this constant $C$ is in fact independent of the choice of $z_0$,
    of $f$ and of the various branches used in the construction. It is now clear that
    \[
      F_{\Cp}(z;1) \geq c,
    \]
    for all $z$ in some neighborhood $U$ of $z_0$ and for some positive constant $c$. This shows that the Kobayashi metric is bounded from below in a neighborhood
    of $z_0$ and hence that $\Cp$ is hyperbolic. For the sake of completeness (no pun intended), let us make the bound above more useable. We have for
    $|z| < 1, z \neq 0$,
    \begin{align*}
      |f'(0)| &\leq C|z|\left|\sqrt{(\log z)^2 - 1}\right|\leq C|z|\left|\sqrt{(\log z + 1)(\log z - 1)}\right|\\
      &\leq C'|z||\log z|.
    \end{align*}
    Thus, on the punctured disk $\disk^*$, we have
    \[
      F_{\Cp}(z;1) \geq \frac{1}{C'|z||\log z|}.
    \]
    
    It remains to show that the Kobayashi distance on $\Cp$ is complete. This actually follows from the argument in Theorem~\ref{T:compzalc} but we will give another one using the above bound.
    Let $x_n$ be a sequence in $\Cp$ that converges to $0$.
    We will show that $x_n$ is \textbf{not} a bounded sequence with respect to
    the Kobayashi distance. To see this, note that for any fixed point $a \in
    \Cp$, we have
     \[
     K_{\Cp}(z,a) = \inf \left\{ \int_0^1 F_{\Cp}(\gamma(t); \gamma'(t)) dt \right\}, 
  \] 
  where the infimum is taken over all piecewise $C^1$-smooth curves
  $\gamma:[0,1] \to \Cp$ with $\gamma(0)=a$ and $\gamma(1)=z$. We can
  find a branch of the logarithm $\log \gamma(t)$ along the curve $\gamma$. Therefore, using the lower bound for the Kobayashi metric derived above, we have
  \[
  F_{\Cp}(\gamma(t); \gamma'(t)) \geq \frac{|\gamma'(t)|}{C' |\gamma(t)\log \gamma(t)|},\quad \text{for all } t \in [0,1].
  \] 
  Therefore, the Fundamental Theorem of Calculus implies that
  \begin{align*}
    \int_0^1 F_{\Cp}(\gamma(t); \gamma'(t)) dt &\geq \frac{1}{C'} \int_0^1 \frac{|\gamma'(t)|}{|\gamma(t) \log \gamma(t)|} dt \geq \frac{1}{C'} \left| \log \left| \frac{\log z}{\log a} \right| \right| \to \infty, \quad \text{as } z \to 0.
  \end{align*}  
  Thus, the sequence $x_n$ is unbounded with
  respect to the Kobayashi distance. As the automorphism group of $\Cp$ acts
  transitively on the set $\{0,1,\infty\}$, the same argument holds for any sequence converging to any
  boundary point of $\Cp$. This shows that the Kobayashi distance on
  $\Cp$ is complete. 
\end{proof}

\section{A grand equivalence}

The following theorem is inspired by the main result in \cite{Hahn1980}. We
remark that the crucial difference between our result and that of
\cite{Hahn1980} is that we do \textbf{not} use the Uniformization theorem or the
elliptic modular function in our proof. We first need some definitions.

\begin{definition}
  Let $D \subset \C$ be a domain. We say that $D$ has
  \begin{enumerate}
    \item the \textbf{Picard property} (or that $D$ is \textbf{Brody hyperbolic}) if every holomorphic map $f:\C \to D$ is constant.
    \item the \textbf{Schottky property} if for every $0<r<1$ there exists $M(r)>0$ such that any $f \in \hol(\D,D)$ satisfies $|f(z)| \leq M(r)$ for all $|z| \leq r$.
    \item the \textbf{Landau property} if for each $a \in D$, we can find a
    constant $R(a) > 0$ such that for any holomorphic map $f:\D(0,R(a)) \to \C$ with
    $f(0) = a$ and $|f'(0)| = 1$, the image $f(\D(0,R(a)))$ is \textbf{not} fully
   contained in $D$.
   \item the \textbf{Montel property} if the family $\hol(\D,D)$ is normal. 
  \end{enumerate}
\end{definition}

\begin{theorem}
  Let $D \subset \C$ be a domain. The following statements are equivalent:
  \begin{enumerate}
    \item The complement $\C \setminus D$ contains at least two points.
    \item The domain $D$ is Kobayashi hyperbolic.
    \item $D$ is complete Kobayashi hyperbolic.
    \item $D$ has the Landau property.
    \item $D$ has the Schottky property.
    \item $D$ has the Picard property.
    \item $D$ has the Montel property.  
\end{enumerate}
\end{theorem}

\begin{proof}
  Two proofs of $(1) \implies (2)$ have already been provided, and the fact that $(2) \implies (3)$ 
  is part of the proof of Theorem~\ref{T:compzalc}.  We have established $(3) \implies (4)$ and  $(3) \implies (5)$ in the proofs of Landau's and Schottky's theorems in Section~4. The implication $(5) \implies (6)$ is standard; see, for example, \cite[Theorem~3, Chapter~4]{Narasimhan2001}. 
   The implications $(4) \implies (1)$ and $(6) \implies (1)$ are obvious by considering the exponential map. The implication $(1) \implies  (7)$ follows from Montel's theorem. 
  It remains to show that (7) $\Rightarrow$ (1): Assume that $D$ does not satisfy (1). Then $\C \setminus D$ contains at most one point. Without loss of generality, we may assume that $D$ is either $\C$ or $\C \setminus \{0\}$. In both cases, it is clear that the family $\hol(\D,D)$ is not normal, as the sequence of holomorphic maps $f_n(z) = nz$ in $\hol(\D,\C)$ and the sequence of holomorphic maps $g_n(z) = e^{nz}$ in $\hol(\D,\C\setminus\{0\})$ do not have any convergent subsequences. This shows that $D$ does not satisfy (7).
\end{proof}

 \bibliographystyle{amsalpha}

\bibliography{Kob_hyperbolic}



\end{document}